\documentclass[11pt,a4paper]{article}

\usepackage[T1]{fontenc}
\usepackage{amsmath,amssymb,amsthm,mathtools,mathrsfs}
\usepackage[a4paper,margin=1in]{geometry}
\IfFileExists{newtxtext.sty}{%
  \usepackage{newtxtext,newtxmath}%
}{%
  \usepackage{mathptmx}%
}
\DeclareFontFamily{U}{rsfs}{\skewchar\font127}
\DeclareFontShape{U}{rsfs}{m}{n}{<-6> rsfs5 <6-8> rsfs7 <8-> rsfs10}{}
\usepackage{microtype}
\usepackage{xcolor}
\usepackage[colorlinks=true,linkcolor=red!75!black,citecolor=green!45!black,urlcolor=blue!65!black]{hyperref}

\hypersetup{
  pdftitle={On Integer-Division Bases for the Kalmar Elementary Functions},
  pdfauthor={Joseph M. Shunia},
  pdfsubject={Fixed-base incompleteness, variable-base completeness, and elementary predicates},
  pdfkeywords={arithmetic term, substitution basis, elementary function, integer division, exponentiation, geometric sparsity, elementary predicates}
}

\newcommand{\N}{\mathbb N}
\newcommand{\E}{\mathcal E}
\newcommand{\Pow}{\mathscr P}
\newcommand{\sg}{\operatorname{sg}}
\newcommand{\dotminus}{\mathbin{\dot-}}
\newcommand{\HW}{\operatorname{HW}}
\newcommand{\Pred}{\operatorname{Pred}}
\newcommand{\pred}{\operatorname{pred}}

\newtheorem{theorem}{Theorem}
\newtheorem{lemma}[theorem]{Lemma}
\newtheorem{corollary}[theorem]{Corollary}
\theoremstyle{remark}
\newtheorem{remark}[theorem]{Remark}

\allowdisplaybreaks[1]

\title{On Integer-Division Bases for the Kalm\'ar Elementary Functions}
\author{Joseph M. Shunia}
\date{October 2026}

\begin{document}
\maketitle

\begin{abstract}\noindent
Prunescu, Sauras-Altuzarra, and Shunia asked whether addition, totalized
integer division, and base-two exponentiation form a substitution basis for
the Kalm\'ar elementary functions. We answer this problem negatively, even
after adjoining multiplication. On independent powers of two, every fixed
term built from addition, multiplication, and integer division takes only
polylogarithmically many distinct values below a given bound, and the same
restriction survives arbitrary nesting of $2^x$. No such term recovers $n$
from $2^n$, so neither $\lfloor\log_2 x\rfloor$ nor integer remainder is
representable. The obstruction is not special to the base two. For every
fixed integer $b\geq 2$, the same sparsity bound holds with constants
independent of $b$, and
$\langle x+y,\lfloor x/y\rfloor,b^x\rangle$
remains a proper subclass of the elementary functions. Variable
exponentiation reverses the conclusion: addition, integer division, and
$x^y$ do form a substitution basis,
\[
  \langle x+y,\ \lfloor x/y\rfloor,\ x^y\rangle=\E.
\]
Once the base may depend on the input, integer remainder and truncated
subtraction become representable. The incomplete base-two class is
nevertheless complete for truth values. Every elementary characteristic
function belongs to it, as does every elementary function with a fixed
finite range.
\end{abstract}



\section{Introduction}\label{sec:intro}

In this work, $\N$ denotes the set of non-negative integers. Given a set
$B$ of operations on $\N$, we write $\langle B\rangle$ for the closure under
substitution of $B$ together with the constant operations and projections.
We say that $B$ is a \emph{substitution basis} for a class $\mathcal G$ of
operations if $\langle B\rangle=\mathcal G$; see, for example,
Mazzanti~\cite[Section~2.2]{Mazzanti2002}. Thus only finite composition is
allowed: input-dependent iteration, search, and variable-length sums or
products are not term operations.

The Kalm\'ar elementary functions form the class $\mathcal E^3$ of the
Grzegorczyk hierarchy~\cite{Grzegorczyk1953}; following~\cite{PSS2026}, we
denote this class simply by $\E$. Equivalently, $\E$ is the smallest class
containing the constants, projections, addition, and truncated subtraction
$x\dotminus y=\max\{x-y,0\}$ and closed under substitution, bounded summation,
and bounded multiplication~\cite{Clote1999,Cutland1980}. The problem of
finding finite substitution bases for $\E$ goes back at least to
R\"odding~\cite{Rodding1964} and Marchenkov~\cite{Marchenkov1980}.
Jones and Matiyasevich~\cite{JonesMatiyasevich1982} studied Kalm\'ar
elementary relations in connection with exponential Diophantine
representations and obtained substitution bases for their characteristic
functions; Jones~\cite{Jones1989} later treated the function class itself.
Mazzanti proved the arithmetic basis~\cite{Mazzanti2002}
\[
  \langle x+y,\ x\dotminus y,\ xy,\ \lfloor x/y\rfloor,\ x^y\rangle=\E,
\]
and Marchenkov later obtained, among other results, the basis
~\cite{Marchenkov2007}
\begin{equation}\label{eq:marchenkov-basis}
  \langle x+y,\ x\bmod y,\ x^2,\ 2^x\rangle=\E.
\end{equation}

Prunescu, Sauras-Altuzarra, and Shunia~\cite{PSS2026} sharpened
\eqref{eq:marchenkov-basis} by eliminating squaring and proving the minimal
substitution basis
\begin{equation}\label{eq:remainder-basis}
  \langle x+y,\ x\bmod y,\ 2^x\rangle=\E.
\end{equation}
They also proved that integer division is representable in this basis
~\cite[Theorem~4]{PSS2026}. This naturally raises the converse question:
can integer division replace integer remainder? At the end of their paper
they stated this explicitly as Problem~15~\cite[Problem~15]{PSS2026}:
\begin{equation}\label{eq:open-problem}
  \text{Is }\quad
  \langle x+y,\ \lfloor x/y\rfloor,\ 2^x\rangle=\E\ ?
\end{equation}
We answer this problem negatively and show that replacing $2$ by any other
fixed integer base does not change that answer. Allowing the base itself
to vary does change it: addition, integer division, and binary
exponentiation $x^y$ form a substitution basis for $\E$. We also determine
the expressive power of the original base-two quotient class at the level
of predicates. It omits elementary numerical functions but represents
every elementary truth value.

Throughout, $\lfloor a/b\rfloor$ denotes integer division, totalized by the
convention $\lfloor a/0\rfloor=0$, as in~\cite{Mazzanti2002,PSS2026}.
We distinguish fixed-base exponentiation $b^a$ from the binary operation
$a^c$, for which we use $0^0=1$. Constants and projections are always
available. A fixed straight-line program may be
used in place of a term, since shared intermediate expressions can be
unfolded into a finite term tree.

We write
\[
  S=\langle x+y,\lfloor x/y\rfloor,2^x\rangle,
  \qquad
  \mathcal T=\langle x+y,\lfloor x/y\rfloor,x^y\rangle.
\]
For a class $\mathcal C$ of functions, $\Pred(\mathcal C)$ denotes the
predicates, of arbitrary finite arity, whose $\{0,1\}$-valued characteristic
functions belong to $\mathcal C$. Our three main results are as follows.

\begin{theorem}[Function incompleteness]\label{thm:main}
Multiplication belongs to $S$, and
\[
  S=\langle x+y,xy,\lfloor x/y\rfloor,2^x\rangle\subsetneq\E.
\]
No unary term $T\in S$ satisfies
\begin{equation}\label{eq:decoder}
  T(2^n)=n\qquad(n\in\N).
\end{equation}
In particular, neither integer logarithm nor integer remainder belongs to
$S$, and the answer to Problem~15 of Prunescu, Sauras-Altuzarra, and Shunia
is negative.
\end{theorem}

The same obstruction holds for every fixed integer base $b\geq2$.
Theorem~\ref{thm:fixed-base} proves the corresponding sparsity estimate
uniformly in $b$, and Corollary~\ref{cor:fixed-base-incomplete} gives the
negative basis result.

\begin{theorem}[Variable-base completeness]\label{thm:variable-basis}
Addition, totalized integer division, and binary exponentiation form a
substitution basis for the Kalm\'ar elementary functions:
\[
  \mathcal T=\langle x+y,\lfloor x/y\rfloor,x^y\rangle=\E.
\]
The translation from an elementary description to a $\mathcal T$-term is
effective. In particular, integer remainder and truncated subtraction
belong to $\mathcal T$.
\end{theorem}

\begin{theorem}[Predicate completeness]\label{thm:predicates}
Every elementary predicate has its characteristic function in $S$:
\[
  \Pred(S)=\Pred(\E).
\]
The translation from a finite elementary description to an $S$-term is
effective. More generally, every elementary function with a fixed finite
range belongs to $S$.
\end{theorem}

The first theorem does not arise from an inability to perform ordinary
multiplication: adjoining that operation leaves the class unchanged.
Section~\ref{sec:operations} also gives terms for predecessor and
exponentiated remainder. Thus
\[
  x\dotminus1\in S,\qquad 2^{x\bmod y}\in S,
  \qquad x\bmod y\notin S.
\]
The class can represent an injective exponential encoding of a residue,
but cannot uniformly recover the residue. Theorem~\ref{thm:variable-basis}
shows that variable-base exponentiation supplies enough freedom to recover
all elementary numerical outputs. Theorem~\ref{thm:predicates} describes a
different feature of the weaker class: every elementary yes-or-no
question, including the graph of remainder, has a characteristic term in
$S$ even though the corresponding numerical output need not have such a
term.

The negative proof is based on a sparsity phenomenon for dyadic inputs. Put
\[
  \Pow=\{2^e:e\in\N\},\qquad
  g(M)=1+\log_2(M+1)\quad(M\geq1).
\]
We prove that every fixed term $A$ using only $+$, $\times$, and $\lfloor x/y\rfloor$
satisfies
\begin{equation}\label{eq:intro-sparsity}
  \#\bigl(A(\Pow^k)\cap\{0,\ldots,M\}\bigr)
  \leq C_A g(M)^{K_A},
\end{equation}
with the input exponents independent and unrestricted. The difficulty is
integer division: a small quotient can have arbitrarily large operands, so
a naive induction on the ranges of the operands fails.

The key idea is to separate small and large intermediate quotients. Small
quotients are frozen at their exact values and counted inductively. After
that reduction, every surviving proper quotient is large, so its floor may
be replaced by real division with a uniformly controlled logarithmic error.
For each frozen configuration the resulting expression is a positive
rational function with bounded algebraic complexity, and a normalization of
its dyadic monomials gives a polylogarithmic approximation set. This proves
\eqref{eq:intro-sparsity}. Finally, arbitrary nesting of $2^x$ is removed by
replacing outermost exponential subterms by independent dyadic variables.
The resulting polylogarithmic range cannot contain all integers
$0,1,\ldots,M$, which rules out a decoder satisfying~\eqref{eq:decoder}.

The variable-base proof has two stages. First, an elementary function
$f(\mathbf x)$ can be encoded as $c^{f(\mathbf x)}$ by a $\mathcal T$-term
with an additional base argument $c\geq1$. The main closure step is
integer division: a geometric-series quotient, reduced modulo one less
than its base, records the quotient of two encoded exponents. This residue
is itself used only through its exponential encoding. Second, for a large
enough $B$, the sequence $B(1-1/B)^j$ crosses the threshold $f+1$ in a step
that leaves it strictly above $f$. An encoded count locates that step, and
one final integer division recovers $f$. The proof compiles the count
through the established basis~\eqref{eq:remainder-basis}; it does not add
bounded summation or search to the term language.

Predicate completeness uses a separate construction entirely within $S$.
We first show that the double exponential $2^{2^{\HW(x)}}$, where $\HW$ denotes binary Hamming
weight, belongs to $S$. We then express the hypercube integer associated
with an elementary predicate as a difference $A-B$ with $A,B\in S$.
An explicit quotient construction adds a prescribed block of binary ones
to this difference. This produces an $S$-term whose Hamming weight detects
whether the hypercube contains a witness. No decoding of a general
numerical payload is required.

Section~\ref{sec:operations} establishes the basic constructions in $S$.
Sections~\ref{sec:approximation}--\ref{sec:separation} prove the dyadic
negative result, and Section~\ref{sec:fixed-bases} extends it uniformly to
all fixed integer bases. Section~\ref{sec:variable-base} proves the
variable-base basis theorem. Finally, Section~\ref{sec:hamming} constructs
the double Hamming exponential in $S$, and Section~\ref{sec:predicates}
proves the hypercube normal form and predicate completeness. These last
two sections are independent of the variable-base proof.

\section{Basic constructions in the quotient class}\label{sec:operations}

We use $x\bmod0=x$, consistently with
$x\bmod y=x-y\lfloor x/y\rfloor$. Subtraction in an identity is an ordinary
integer operation used in its proof; it is not an additional generator of
$S$. Every claimed $S$-term below is given using the permitted operations.

\begin{lemma}[Multiplication]\label{lem:multiplication}
For all $x,y\in\N$,
\begin{equation}\label{eq:multiplication}
  xy=\left\lfloor
  \frac{2^{2x+2y+2}}
  {\left\lfloor\left\lfloor2^{2x+2y+2}/x\right\rfloor/y\right\rfloor}
  \right\rfloor.
\end{equation}
Consequently $S=\langle x+y,xy,\lfloor x/y\rfloor,2^x\rangle$.
\end{lemma}

\begin{proof}
For positive integers $b,c$,
$\lfloor\lfloor a/b\rfloor/c\rfloor=\lfloor a/(bc)\rfloor$.
Also, if $m>0$ and $N>m(m+1)$, then
\[
  m\leq\frac{N}{\lfloor N/m\rfloor}
  <\frac{mN}{N-m}<m+1.
\]
Thus $\lfloor N/\lfloor N/m\rfloor\rfloor=m$.
For $x,y>0$, take $N=2^{2x+2y+2}$ and $m=xy$. Since
$xy\leq2^{x+y}$, we have $xy(xy+1)<N$, proving
\eqref{eq:multiplication}. If either input is zero, the inner denominator
and the totalized outer quotient are both zero, as required.
\end{proof}

\begin{lemma}[Predecessor, tests, and encoded remainder]\label{lem:basic}
The following functions belong to $S$:
\begin{align}
  \pred(x):=x\dotminus1
    &=\left\lfloor\frac{x^2}{x+1}\right\rfloor,\label{eq:pred}\\
  \delta_0(x):=\boldsymbol1_{x=0}
    &=\left\lfloor\frac1{x+1}\right\rfloor,\qquad
  \sg(x):=\boldsymbol1_{x>0}
    =\left\lfloor\frac{2x}{x+1}\right\rfloor,\label{eq:tests}\\
  F(x,y):=2^{x\bmod y}
    &=\left\lfloor\frac{2^x}{2^{y\lfloor x/y\rfloor}}\right\rfloor.
    \label{eq:encoded-remainder}
\end{align}
\end{lemma}

\begin{proof}
For $x>0$, $x^2=(x-1)(x+1)+1$; the case $x=0$ is immediate.
The two tests follow directly from their displayed quotients. Finally,
$y\lfloor x/y\rfloor\leq x$, also when $y=0$, so the last quotient is
exactly $2^{x-y\lfloor x/y\rfloor}$.
\end{proof}

\begin{remark}[Nested remainder encodings]\label{rem:nesting}
Membership of $F$ in $S$ permits substitution of $S$-terms for its
arguments; it does not permit substitution of an unavailable numerical
remainder. A useful identity that respects this distinction is
\[
  \bigl((2^x-1)\bmod(2^a-1)\bigr)+1=2^{x\bmod a}
  \qquad(x,a\in\N).
\]
For $a>0$, write $x=qa+r$, $0\leq r<a$, and use
$2^x-1=(2^a-1)2^r\sum_{j=0}^{q-1}2^{aj}+2^r-1$.
For $a=0$, use the convention for remainder by zero. Hence
\begin{equation}\label{eq:nested-encoding}
  2^{\,2^{x\bmod(y\bmod z)}}
  =2F\bigl(\pred(2^x),\pred(F(y,z))\bigr)\in S.
\end{equation}
This establishes a double exponential encoding, not a single exponential
one. Predicate completeness below does not assume that
$2^{x\bmod(y\bmod z)}$ belongs to $S$.
\end{remark}

\section{Two approximation lemmas}\label{sec:approximation}

The first lemma bounds the number of approximations needed for a positive
rational function on dyadic inputs. The second controls the relative error
introduced by removing floors at large intermediate values. Real division
and logarithms are used only in the analysis.

\paragraph{Positive rational functions on dyadic inputs.}

\begin{lemma}\label{lem:rational}
Let $P,Q$ be nonzero polynomials in $k$ variables with nonnegative integer
coefficients. Suppose each has at most $J$ nonzero monomials and every
coefficient is at most $B$, where $J$ is a positive integer and $B\geq1$.
For every $M\geq1$
there is a finite set $\mathcal F\subset(0,\infty)$ such that every value
\[
  R=\frac{P(\mathbf z)}{Q(\mathbf z)},
  \qquad \mathbf z\in\Pow^k,\qquad \frac12\leq R\leq M+2,
\]
lies within $1/100$ of a member of $\mathcal F$, and
\begin{equation}\label{eq:net-asymptotic}
  \#\mathcal F
  \leq C_J\bigl(1+\log_2((M+2)B)\bigr)^{2J+1},
\end{equation}
where $C_J$ depends only on $J$. In particular, if
$B\leq c(M+1)^d$ for fixed $c\geq1$ and $d\in\N$, then
$\#\mathcal F\leq Cg(M)^{2J+1}$, with $C$ depending only on $J,c,d$.
\end{lemma}

\begin{proof}
Write $\mathbf z=(2^{e_1},\ldots,2^{e_k})$. Then
\[
  P(\mathbf z)=\sum_{i=1}^{p}c_i2^{\alpha_i\cdot\mathbf e},
  \qquad
  Q(\mathbf z)=\sum_{j=1}^{q}d_j2^{\beta_j\cdot\mathbf e},
\]
where $p,q\leq J$, the coefficients belong to
$\{1,\ldots,\lfloor B\rfloor\}$, and the multiindices have nonnegative
integer entries. Set
\[
  E_P=\max_i\alpha_i\cdot\mathbf e,\qquad
  E_Q=\max_j\beta_j\cdot\mathbf e,\qquad
  \Delta=E_P-E_Q\in\mathbb Z,
\]
and normalize by
\[
  \overline P=2^{-E_P}P(\mathbf z),\qquad
  \overline Q=2^{-E_Q}Q(\mathbf z).
\]
Each normalized sum contains an exponent-zero monomial with coefficient at
least one. Hence
\begin{equation}\label{eq:normal-bounds}
  1\leq\overline P,\overline Q\leq JB,
  \qquad R=2^\Delta\frac{\overline P}{\overline Q}.
\end{equation}
The assumed range of $R$ gives
\begin{equation}\label{eq:delta-bounds}
  -\log_2(2JB)\leq\Delta\leq\log_2((M+2)JB).
\end{equation}

Choose
\[
  L=\left\lceil\log_2\bigl(200(M+2)(JB)^3\bigr)\right\rceil.
\]
In each normalized sum, discard monomials whose exponents are less than
$-L$, and denote the retained sums by $P_*,Q_*$. The exponent-zero monomials
remain, so $P_*,Q_*\geq1$. With $\epsilon=JB2^{-L}$, we have
\[
  0\leq\overline P-P_*\leq\epsilon,
  \qquad 0\leq\overline Q-Q_*\leq\epsilon.
\]
Since $\overline Q,Q_*\geq1$ and $P_*\leq JB$,
\[
  \left|\frac{\overline P}{\overline Q}-\frac{P_*}{Q_*}\right|
  \leq(1+JB)\epsilon.
\]
Indeed, the numerator of the difference is
$(\overline P-P_*)Q_*-P_*(\overline Q-Q_*)$. Thus the approximation
$R_*=2^\Delta P_*/Q_*$ satisfies
\begin{align*}
  |R-R_*|
  &\leq2^\Delta(1+JB)JB2^{-L}\\
  &\leq2(M+2)(JB)^3\,2^{-L}
  \leq\frac1{100},
\end{align*}
by~\eqref{eq:delta-bounds} and the choice of $L$.

It remains to count the possible approximations. For each of the at most
$2J$ monomials, record either its normalized exponent in
$\{-L,\ldots,0\}$ or the symbol ``discarded.'' There are $L+2$ possible
records per monomial. Together with $\Delta$, these records determine
$R_*$, since the coefficient lists are fixed. Take all records with a
retained exponent-zero monomial in each sum and all integer values of
$\Delta$ allowed by~\eqref{eq:delta-bounds}. They form a finite set
$\mathcal F$ with
\begin{equation}\label{eq:net-count}
  \#\mathcal F
  \leq\bigl(3+\log_2(2(M+2)(JB)^2)\bigr)(L+2)^{2J}.
\end{equation}
Records need not be realized by an input tuple; including them only
enlarges the set. Both $L+2$ and the first factor in~\eqref{eq:net-count}
are bounded by a constant depending on $J$ times
$1+\log_2((M+2)B)$. This proves~\eqref{eq:net-asymptotic} and its
polynomial-coefficient specialization.
\end{proof}

\paragraph{Stability under removal of large floors.}

\begin{lemma}\label{lem:rounding}
Consider a finite term tree built from $+$, $\times$, and $\lfloor x/y\rfloor$, evaluated
at nonnegative integer inputs. Some whole subtrees may already have been
replaced by their exact values. Suppose every remaining division node has
integer output greater than $H>0$. Replace these divisions by ordinary real
division. The resulting evaluation is well-defined and preserves which
node values are zero. At every node with positive integer value $v$, its
new value $\widetilde v$ satisfies
\begin{equation}\label{eq:log-error}
  \left|\ln\frac{\widetilde v}{v}\right|\leq\frac rH,
\end{equation}
where $r$ is the number of remaining division occurrences in its subtree.
\end{lemma}

\begin{proof}
We induct on the tree. Leaves and exactly replaced subtrees have zero
error. Addition and multiplication of nonnegative quantities preserve zero
patterns. At a division node with output $d>H$, both operands are positive;
their approximate values are therefore positive by induction, so the new
division is defined.

For positive products and quotients, absolute logarithmic errors are
bounded by the sum of the operand errors. For a positive sum, the ratio of
its approximate value to its exact value is a weighted average of the
corresponding ratios for its positive summands. Its logarithmic error is
therefore at most the larger summand error, and hence at most their sum.
Zero summands remain exactly zero.

At a division node, the additional local error comes from the floor.
Writing $d=\lfloor a/b\rfloor>H$, we have
\[
  d\leq a/b<d+1,
  \qquad
  0\leq\ln\frac{a/b}{d}
    <\ln(1+1/d)\leq\frac1d<\frac1H.
\]
Consequently,
\[
  \left|\ln\frac{\widetilde a/\widetilde b}{d}\right|
  \leq
  \left|\ln\frac{\widetilde a}{a}\right|
  +\left|\ln\frac{\widetilde b}{b}\right|+\frac1H.
\]
Counting the division occurrences in the two child subtrees and the
current node gives~\eqref{eq:log-error} and completes the induction.
\end{proof}

\section{Dyadic sparsity}\label{sec:sparsity}

For a term $A(z_1,\ldots,z_k)$ using $+$, $\times$, and $\lfloor x/y\rfloor$, write
\[
  N_A(M)=\#\bigl(A(\Pow^k)\cap\{0,\ldots,M\}\bigr).
\]
The next lemma controls this range count even though the input exponents
may be arbitrarily large.

\begin{lemma}[Dyadic sparsity]\label{lem:sparsity}
For every such term $A$, with arbitrary natural-number constants allowed,
there are constants $C_A\geq1$ and $K_A\in\N$ such that
\begin{equation}\label{eq:sparsity}
  N_A(M)\leq C_Ag(M)^{K_A}
  \qquad(M\in\N,\ M\geq1).
\end{equation}
\end{lemma}

\begin{proof}
We use structural induction on the term tree. A constant has at most one
value, and a variable has at most $1+\lfloor\log_2M\rfloor$ values below
$M$. If $A=U+V$ has output at most $M$, then both child outputs are at most
$M$, so
\[
  N_{U+V}(M)\leq N_U(M)N_V(M).
\]
Likewise, a positive product at most $M$ has both factors in
$\{1,\ldots,M\}$; accounting separately for zero gives
\[
  N_{UV}(M)\leq1+N_U(M)N_V(M).
\]
The induction hypothesis therefore settles addition and multiplication.

Now let $A=\lfloor U/V\rfloor$, and let $s\geq1$ be the number of operation
occurrences in its full term tree. Fix $M\geq1$ and set
\begin{equation}\label{eq:H}
  H=100(s+1)(M+1)^2.
\end{equation}
Zero contributes at most one output, so consider an input tuple for which
$q=A(\mathbf z)\in\{1,\ldots,M\}$. If $u,v$ are the exact values of the
root's children, then $u,v>0$ and
\begin{equation}\label{eq:x}
  q\leq x:=\frac uv<q+1\leq M+1.
\end{equation}

\smallskip
\noindent\emph{Freezing the small quotients.}
Among the proper division nodes with output at most $H$, select those with
no such ancestor below the root. Replace each selected subtree by its
exact value; we call these replacements \emph{cuts}. Every remaining
proper division node has output greater than $H$, since any unselected
small node lies below a selected one. The root itself is never cut.

There are at most $2^s$ cut patterns. Each cut subtree is a strictly smaller
term than $A$, so the induction hypothesis bounds its possible values in
$\{0,\ldots,H\}$ by $C_i g(H)^{K_i}$. For a fixed pattern, multiplying
these bounds controls the number of possible tuples of cut values,
regardless of dependencies among them. Since $H$ is polynomial in $M$,
$g(H)=O_A(g(M))$. Thus the number of cut patterns with assigned values is
at most
\begin{equation}\label{eq:configs}
  C_0g(M)^{K_0}
\end{equation}
for constants depending only on $A$. We retain only configurations realized
by an input tuple under consideration. All proper divisions with zero
output, including those with zero denominator, have been removed by the
cuts.

\smallskip
\noindent\emph{Removing the remaining floors.}
Fix a realized configuration. Replace all remaining divisions, including
the root, by real division, and write the resulting root quotient as
$R=\widetilde u/\widetilde v$. Lemma~\ref{lem:rounding}, applied to the two
child trees, preserves their positivity and gives
\[
  \left|\ln\frac Rx\right|
  \leq\frac sH=:\eta
  \leq\frac1{100(M+1)^2}.
\]
The root floor is not part of this error: $R$ is compared with the real
quotient $x=u/v$, not with its floor $q$. Using
$e^\eta-1\leq2\eta$ for $0\leq\eta\leq1$ and~\eqref{eq:x}, we obtain
\begin{equation}\label{eq:abs-error}
  |R-x|
  \leq x(e^\eta-1)
  \leq2(M+1)\frac sH
  <\frac1{50(M+1)}\leq\frac1{100}.
\end{equation}
In particular, $1/2\leq R\leq M+2$.

\smallskip
\noindent\emph{Uniform control of the rational functions.}
For each fixed configuration, the resulting expression has the form
\begin{equation}\label{eq:pq}
  R(\mathbf z)=\frac{P(\mathbf z)}{Q(\mathbf z)},
\end{equation}
where $P,Q$ have nonnegative integer coefficients. Indeed, start with
variables and constants over denominator one and repeatedly use
\begin{align*}
  \frac{P_1}{Q_1}+\frac{P_2}{Q_2}
    &=\frac{P_1Q_2+P_2Q_1}{Q_1Q_2},\\
  \frac{P_1}{Q_1}\frac{P_2}{Q_2}
    &=\frac{P_1P_2}{Q_1Q_2},\\
  \frac{P_1/Q_1}{P_2/Q_2}
    &=\frac{P_1Q_2}{Q_1P_2}.
\end{align*}
No cancellation is needed. An intermediate numerator may be the zero
polynomial, but a denominator introduced by division is nonzero: its
corresponding value is positive at a realizing tuple. A nonzero polynomial
with nonnegative coefficients is positive at every strictly positive tuple,
so this rational representation is legitimate. The root numerator is
nonzero as well, since $R>0$.

The complexity of $P,Q$ is uniform over configurations. To see this, first
treat the cut values as additional indeterminates. For each of the finitely
many cut patterns, the displayed rules produce a fixed pair of polynomials.
Their coefficients, when viewed as polynomials in the original variables,
are fixed polynomials with nonnegative integer coefficients in the cut
values. Consequently there are constants $J\geq1$, $c\geq1$, and $d\in\N$,
depending only on $A$, such that $P,Q$ each have at most $J$ monomials and
all coefficients are bounded by
\begin{equation}\label{eq:B}
  B(M)=c(H+1)^d.
\end{equation}
One may take $d$ to be the largest degree of a coefficient polynomial and
$c$ to bound its coefficient sum, over all cut patterns. Vanishing
coefficients are simply omitted.

\smallskip
\noindent\emph{Counting the root values.}
Since $B(M)$ is polynomially bounded in $M$, Lemma~\ref{lem:rational}
provides at most $C_1g(M)^{2J+1}$ approximation values for each fixed
configuration, with $C_1$ uniform over all configurations. The coefficient
lists are fixed within a configuration; their possible choices have already
been counted in~\eqref{eq:configs}. The bound~\eqref{eq:B} controls
approximation accuracy, not the enumeration of coefficient lists.

For every relevant input tuple, an approximation $R_*$ satisfies
\[
  |R-R_*|\leq\frac1{100},\qquad
  |x-R_*|<\frac1{50},\qquad q\leq x<q+1.
\]
Therefore
\[
  q\in\{\lfloor R_*\rfloor-1,\ \lfloor R_*\rfloor,
          \ \lfloor R_*\rfloor+1\}.
\]
Each approximation accounts for at most three integer outputs. Combining
this with~\eqref{eq:configs} and restoring the possible zero output gives
\[
  N_A(M)\leq1+3C_0C_1g(M)^{K_0+2J+1}.
\]
This has the required form and completes the structural induction.
\end{proof}

\section{Arbitrary exponential nesting and the negative answer}\label{sec:separation}

We now return to the basis problem posed in~\cite[Problem~15]{PSS2026}.
Dyadic sparsity survives the addition of base-two exponentiation, however
deeply nested, because every exponential subterm takes a dyadic value.

\begin{corollary}\label{cor:exp-sparsity}
For every unary term $T$ built from natural-number constants, $+$,
$\times$, $\lfloor x/y\rfloor$, and $2^x$, there are $C\geq1$ and $K\in\N$ such that
\begin{equation}\label{eq:exp-sparsity}
  \#\bigl(\{T(2^n):n\in\N\}\cap\{0,\ldots,M\}\bigr)
  \leq Cg(M)^K
  \qquad(M\in\N,\ M\geq1).
\end{equation}
\end{corollary}

\begin{proof}
In the term tree of $T$, replace every outermost exponential occurrence by
a fresh variable: these are precisely the exponential nodes with no
exponential ancestor. Each replacement removes its whole subtree,
including any nested exponentials. Replace all remaining occurrences of the
original input by a variable $z_0$. The resulting term
$A(z_0,z_1,\ldots,z_k)$ uses only $+$, $\times$, and $\lfloor x/y\rfloor$.

On input $2^n$, both the original input and all replaced exponential values
belong to $\Pow$. Hence
\[
  \{T(2^n):n\in\N\}\subseteq A(\Pow^{k+1}).
\]
Allowing the new variables to range independently can only enlarge the
range. Lemma~\ref{lem:sparsity} proves the claim.
\end{proof}

\begin{proof}[Proof of Theorem~\ref{thm:main}]
If $T(2^n)=n$ for every $n\in\N$, its range on dyadic inputs contains every
natural number. Corollary~\ref{cor:exp-sparsity} would give
\[
  M+1\leq C\bigl(1+\log_2(M+1)\bigr)^K
  \qquad(M\geq1).
\]
Taking $M=2^t$ yields $2^t+1\leq C(t+2)^K$, a contradiction for large $t$.
For completeness, if $m>K$ and $t\geq2m$, then
\[
  2^t\geq\binom tm\geq\frac{(t/2)^m}{m!},
\]
and the last expression divided by $(t+2)^K$ tends to infinity.

To establish the strict inclusion in $\E$, we verify elementary
membership directly. Define
\[
  \delta_0(a)=1\dotminus a,\qquad
  \sg(a)=1\dotminus\delta_0(a),\qquad
  \chi(a\leq b)=\delta_0(a\dotminus b).
\]
These elementary functions are the zero test, the positivity test, and the
comparison indicator. Bounded sums and products give
\[
  ab=\sum_{j=1}^{b}a,\qquad
  2^a=\prod_{j=1}^{a}2,\qquad
  \lfloor a/b\rfloor=\sg(b)\sum_{j=1}^{a}\chi(jb\leq a).
\]
For $b>0$, the last sum counts the positive multiples of $b$ not exceeding
$a$; for $b=0$, the prefactor makes the value zero. All generators, and
therefore all their finite compositions, belong to $\E$.

Finally, the elementary function
\begin{equation}\label{eq:ell}
  \ell(x)=\sum_{j=1}^{x}\chi(2^j\leq x)
\end{equation}
satisfies $\ell(0)=0$ and $\ell(x)=\lfloor\log_2x\rfloor$ for $x\geq1$.
In particular, $\ell(2^n)=n$: precisely the indices $1,\ldots,n$ contribute,
and $n\leq2^n$. The first part of the theorem excludes a term for $\ell$,
and Lemma~\ref{lem:multiplication} gives the asserted equality of the two
term classes. This proves the strict inclusion. The remainder assertion
follows from Corollary~\ref{cor:remainder} below.
\end{proof}

The negative answer has a further consequence. Prunescu,
Sauras-Altuzarra, and Shunia proved that integer division is representable
from addition, remainder, and base-two exponentiation~\cite[Theorem~4]{PSS2026}.
The converse fails in $S$, which already contains multiplication.

\begin{corollary}\label{cor:remainder}
Integer remainder is not representable by a term built from natural-number
constants, addition, multiplication, integer division, and base-two
exponentiation. Equivalently,
\[
  \bigl((x,y)\longmapsto x\bmod y\bigr)
  \ \notin\ \langle x+y,\ xy,\ \lfloor x/y\rfloor,\ 2^x\rangle.
\]
\end{corollary}

\begin{proof}
If remainder were representable, $S$ would contain the
basis $\{+,\bmod,2^x\}$ of Prunescu, Sauras-Altuzarra, and
Shunia~\cite[Theorem~1]{PSS2026}, and hence all of $\E$. This contradicts the strict inclusion $S\subsetneq\E$ proved above.
\end{proof}

\begin{corollary}\label{cor:valuation-hamming}
Neither the two-adic valuation $\nu_2$, extended by $\nu_2(0)=0$, nor the
binary Hamming-weight function $\HW$ belongs to $S$.
\end{corollary}

\begin{proof}
The identities $\nu_2(2^n)=n$ and $\HW(\pred(2^n))=n$ would give a decoder
for $2^n$ if the respective functions belonged to $S$. Predecessor is
available by Lemma~\ref{lem:basic}, so both conclusions contradict
Theorem~\ref{thm:main}.
\end{proof}

\begin{remark}
The result allows arbitrary fixed constants, fixed straight-line programs,
and arbitrary nesting of $2^x$. The integer exponent differences in
Lemma~\ref{lem:rational} remain available for any one fixed integer base,
as the next section makes precise. They do not constrain an operation
whose base varies with the inputs; Section~\ref{sec:variable-base} proves
completeness after that enlargement. Adjoining truncated subtraction also
recovers $\E$, since it supplies remainder by
$x\bmod y=x\dotminus y\lfloor x/y\rfloor$ and
\eqref{eq:remainder-basis} then applies.
\end{remark}

\section{Every fixed exponential base remains incomplete}\label{sec:fixed-bases}

The negative result does not depend on the choice of base two. In fact,
the sparsity estimate is uniform over all integer bases. For $b\geq2$, put
$\Pow_b=\{b^e:e\in\N\}$, retaining $g(M)=1+\log_2(M+1)$.

\begin{theorem}[Uniform fixed-base sparsity]\label{thm:fixed-base}
For every fixed term $A(z_1,\ldots,z_k)$ formed from natural-number
constants, addition, multiplication, and totalized integer division,
there exist $C_A\geq1$ and $K_A\in\N$ such that
\[
  \#\bigl(A(\Pow_b^k)\cap\{0,\ldots,M\}\bigr)
  \leq C_Ag(M)^{K_A}
\]
for every integer $b\geq2$ and every integer $M\geq1$. The constants are
independent of $b$.
\end{theorem}

\begin{proof}
We first verify the base dependence in Lemma~\ref{lem:rational}. At inputs
$z_i=b^{e_i}$, replace $2^{\alpha\cdot\mathbf e}$ by
$b^{\alpha\cdot\mathbf e}$ and normalize by the largest monomial exponents
as before. The normalized sums still lie in $[1,JB]$, where $B$ is the
coefficient bound, and their leading exponent difference is an integer
$\Delta$ satisfying
\[
  -\log_b(2JB)\leq\Delta\leq\log_b((M+2)JB).
\]
Choose
\[
  L_b=\left\lceil\log_b\bigl(200(M+2)(JB)^3\bigr)\right\rceil.
\]
Discarding normalized exponents below $-L_b$ gives the same approximation
error, since
\[
  |R-R_*|\leq2(M+2)(JB)^3b^{-L_b}\leq\frac1{100}.
\]
The record count becomes
\[
  \bigl(3+\log_b(2(M+2)(JB)^2)\bigr)(L_b+2)^{2J}.
\]
Every logarithm has argument at least one, so $\log_b t\leq\log_2t$
shows that the bound in Lemma~\ref{lem:rational} holds with exactly the
same constant for all $b\geq2$.

Lemma~\ref{lem:rounding} involves only nonnegative integer values and is
unchanged. Now repeat the structural induction of
Lemma~\ref{lem:sparsity} simultaneously for all bases. A variable has at
most $1+\lfloor\log_b M\rfloor\leq1+\lfloor\log_2 M\rfloor$ bounded
values. The addition and multiplication steps have no base dependence.
At a quotient, use the same threshold $H$ from~\eqref{eq:H}. The induction
hypothesis counts cut values uniformly in $b$, and the coefficient bound
\eqref{eq:B} depends only on the term and $H$. Applying the uniform
rational-function estimate just established completes the quotient step,
with constants independent of $b$.
\end{proof}

\begin{corollary}\label{cor:fixed-base-incomplete}
For every fixed integer $b\geq2$,
\[
  S_b:=\langle x+y,\lfloor x/y\rfloor,b^x\rangle
  =\langle x+y,xy,\lfloor x/y\rfloor,b^x\rangle
  \subsetneq\E.
\]
No unary $S_b$-term satisfies $T(b^n)=n$ for all $n\in\N$. Consequently
neither the integer base-$b$ logarithm nor the function
$\nu_b(x)=\max\{e:b^e\mid x\}$ for $x>0$, with $\nu_b(0)=0$, belongs to
$S_b$.
\end{corollary}

\begin{proof}
Replacing $2^{2x+2y+2}$ by $b^{2x+2y+2}$ in
Lemma~\ref{lem:multiplication} gives a term for multiplication, since the
numerator only increases. Abstract the outermost $b$-exponential subterms
of a proposed decoder as independent variables, exactly as in
Corollary~\ref{cor:exp-sparsity}. On input $b^n$, all these variables take
values in $\Pow_b$. Theorem~\ref{thm:fixed-base} would then bound the
decoder's range by a polylogarithmic function, contradicting its containing
every natural number.

All the generators are elementary. The elementary function
\[
  \ell_b(x)=\sum_{j=1}^{x}\chi(b^j\leq x)
\]
equals $\lfloor\log_bx\rfloor$ for $x>0$, is zero at $x=0$, and satisfies
$\ell_b(b^n)=n$. Its absence proves strict inclusion. The identity
$\nu_b(b^n)=n$ similarly excludes $\nu_b$.
\end{proof}

Thus changing one fixed exponential base cannot repair function
incompleteness. Binary exponentiation $x^y$ is a different operation:
its base varies with the inputs, so its values cannot be abstracted as
elements of one fixed set $\Pow_b$.

\section{Variable-base exponentiation and a quotient basis}
\label{sec:variable-base}

The restriction to a fixed exponential base is essential.  In this section
we prove Theorem~\ref{thm:variable-basis}, namely
\[
  \mathcal T:=\langle x+y,\ \lfloor x/y\rfloor,\ x^y\rangle=\E.
\]
The proof first represents elementary functions as exponents, uniformly
in a variable base, and then recovers their values by a decreasing
geometric sequence.  It does not use the Hamming-weight or predicate
constructions below.

Since $S\subseteq\mathcal T$, multiplication and predecessor are available
by Lemmas~\ref{lem:multiplication} and~\ref{lem:basic}.  Products and
$\pred(t)$ will therefore be used as term abbreviations.  We also use the
$\mathcal T$-term
\begin{equation}\label{eq:variable-encoded-remainder}
  \mathcal R(c,u,v)
  :=\left\lfloor
       \frac{c^u}{c^{\,v\lfloor u/v\rfloor}}
     \right\rfloor
  =c^{\,u\bmod v}\qquad(c\geq1).
\end{equation}
This identity includes $v=0$: the exponent in the denominator is then
zero, and our convention is $u\bmod0=u$.  All bases used in this section
are positive, so the convention for $0^0$ is immaterial.

Call a function $f:\N^k\to\N$ \emph{uniformly power-encoded in
$\mathcal T$} if there is a $\mathcal T$-term $F(c,\mathbf x)$ such that
\[
  F(c,\mathbf x)=c^{f(\mathbf x)}
  \qquad(c\geq1,\ \mathbf x\in\N^k).
\]
The argument $c$ is an additional variable, not a fixed base.

\begin{lemma}[Uniform power encoding]\label{lem:uniform-power-encoding}
Every elementary function is uniformly power-encoded in $\mathcal T$.
Given a term for it over the basis~\eqref{eq:remainder-basis}, a uniform
encoding can be constructed effectively.
\end{lemma}

\begin{proof}
Constants and projections have the required encodings.  We establish
closure properties sufficient to translate a term over
$\{+,\bmod,2^x\}$.  Let $F(c,\mathbf x)=c^{f(\mathbf x)}$ and
$G(c,\mathbf x)=c^{g(\mathbf x)}$ be encodings already constructed.
Suppress $\mathbf x$ when no confusion can arise.

Addition and multiplication are encoded, respectively, by
\[
  F(c)G(c),\qquad G(F(c,\mathbf x),\mathbf x).
\]
The latter substitution is legitimate because $F(c,\mathbf x)\geq1$.
For truncated subtraction put $W(c)=\lfloor F(c)/G(c)\rfloor$.  Then
\begin{equation}\label{eq:encoded-truncated-subtraction}
  W(c)+\left\lfloor\frac1{W(c)+1}\right\rfloor
  =c^{\,f\dotminus g}.
\end{equation}
For $c\geq2$, the quotient is zero when $f<g$ and equals $c^{f-g}$
otherwise; for $c=1$, both sides of~\eqref{eq:encoded-truncated-subtraction}
equal $1$.

The main step is closure under totalized integer quotient.  Set
\begin{equation}\label{eq:encoded-quotient-construction}
  B=F(2)+2,\qquad V=\pred(B),\qquad
  Q=\left\lfloor
       \frac{\pred(F(B))}{\pred(G(B))}
     \right\rfloor.
\end{equation}
These are $\mathcal T$-terms in the original arguments.  We claim that
\begin{equation}\label{eq:encoded-quotient-residue}
  Q\bmod V=\lfloor f/g\rfloor.
\end{equation}
If $g=0$, the denominator defining $Q$ is zero, so both sides are zero.
If $g>0$, write $f=qg+r$, where $0\leq r<g$.  The exact geometric-series
identity gives
\[
  \left\lfloor\frac{B^f-1}{B^g-1}\right\rfloor
  =B^r\sum_{i=0}^{q-1}B^{ig}.
\]
Here the sum is empty when $q=0$.  The discarded fractional part is
$(B^r-1)/(B^g-1)$, which lies in $[0,1)$.  Reduction modulo $B-1$ therefore
gives $Q\equiv q\pmod{B-1}$.  Moreover,
\[
  q\leq f<2^f+1=B-1=V,
\]
which proves~\eqref{eq:encoded-quotient-residue}.  Applying
\eqref{eq:variable-encoded-remainder} to the actual term values $Q,V$
now yields the required quotient encoding:
\begin{equation}\label{eq:uniform-encoded-quotient}
  \mathcal R(c,Q,V)=c^{\lfloor f/g\rfloor}.
\end{equation}
The residue in~\eqref{eq:encoded-quotient-residue} is used only to verify
this identity; the constructed term uses its encoded form
\eqref{eq:uniform-encoded-quotient}.

Since
\[
  f\bmod g=f\dotminus\bigl(g\lfloor f/g\rfloor\bigr)
\]
also holds when $g=0$, the preceding constructions encode remainder.
Finally, base-two exponentiation is encoded by
\[
  c^{F(2,\mathbf x)}=c^{\,2^{f(\mathbf x)}}.
\]
Induction over a term in the established
basis~\eqref{eq:remainder-basis} proves the assertion.  Every step is an
explicit finite substitution, so this translation is effective.
\end{proof}

The next elementary observation recovers a value from an encoded count.
The sequence is chosen so that the first step below the threshold $a+1$
remains strictly above $a$.

\begin{lemma}[Geometric recovery]\label{lem:geometric-recovery}
Let $a\in\N$ and let $B\geq a+2$ be an integer.  Define
\[
  n=\#\left\{0\leq j<B^2:
          B(B-1)^j\geq(a+1)B^j
        \right\}.
\]
Then
\begin{equation}\label{eq:geometric-recovery}
  a=\left\lfloor\frac{B(B-1)^n}{B^n}\right\rfloor.
\end{equation}
\end{lemma}

\begin{proof}
Put $t_j=B(1-1/B)^j$.  This sequence strictly decreases and
$t_0=B>a+1$.  Bernoulli's inequality gives
$\bigl(B/(B-1)\bigr)^{B-1}\geq2$, and hence
\[
  t_{B^2}\leq B\,2^{-B}<1\leq a+1.
\]
Consequently the counted indices form the initial segment
$0,\ldots,n-1$, with $1\leq n\leq B^2$, and
\[
  t_n<a+1\leq t_{n-1}.
\]
Multiplication of the second inequality by $1-1/B$ gives
\[
  t_n\geq\left(1-\frac1B\right)(a+1)
      =a+1-\frac{a+1}{B}>a.
\]
Thus $a<t_n<a+1$, proving~\eqref{eq:geometric-recovery}.
The argument includes $a=0$ and $B=2$.
\end{proof}

\begin{proof}[Proof of Theorem~\ref{thm:variable-basis}]
All generators of $\mathcal T$ are elementary, so $\mathcal T\subseteq\E$.
Conversely, fix $f\in\E$.  By
Lemma~\ref{lem:uniform-power-encoding}, choose an encoding
$F(c,\mathbf x)=c^{f(\mathbf x)}$, and set
\[
  B(\mathbf x)=F(2,\mathbf x)+2=2^{f(\mathbf x)}+2.
\]
The first expression shows that $B\in\mathcal T$, without presupposing
$f\in\mathcal T$; the second shows that $B\geq f+2$.
Consider the elementary function
\begin{equation}\label{eq:recovery-count}
  N(\mathbf x)=
  \#\left\{0\leq j<B(\mathbf x)^2:
       B(\mathbf x)\bigl(B(\mathbf x)-1\bigr)^j
       \geq\bigl(f(\mathbf x)+1\bigr)B(\mathbf x)^j
     \right\}.
\end{equation}
It is elementary because its defining inequality is elementary and the
count is a bounded sum.  Apply
Lemma~\ref{lem:uniform-power-encoding} to $N$ to obtain a
$\mathcal T$-term $H(c,\mathbf x)=c^{N(\mathbf x)}$.  Then
Lemma~\ref{lem:geometric-recovery} gives
\begin{equation}\label{eq:final-variable-base-decoder}
  f(\mathbf x)=
  \left\lfloor
    \frac{B(\mathbf x)\,
          H(\pred(B(\mathbf x)),\mathbf x)}
         {H(B(\mathbf x),\mathbf x)}
  \right\rfloor.
\end{equation}
Both bases supplied to $H$ are positive, and its denominator is at least
$1$.  The right-hand side is a $\mathcal T$-term, proving
$\E\subseteq\mathcal T$.

The bounded count in~\eqref{eq:recovery-count} serves to specify the
elementary function whose encoding is required.  Using the established
arithmetic basis, it is converted to a finite term and then translated
by Lemma~\ref{lem:uniform-power-encoding}.  Thus no bounded summation or
search remains as an operation in~\eqref{eq:final-variable-base-decoder}.
The construction is effective, although no claim of a short resulting
term is made.
\end{proof}

\begin{remark}[Recovering integer remainder]\label{rem:remainder-recovery}
For the original remainder question, the recovery construction can use
the simpler bound $B=x+2$. The function
\[
  N(x,y)=\#\left\{0\leq j<B^2:
       (x+1)B^j\leq B(B-1)^j
                    +y\lfloor x/y\rfloor B^j\right\}
\]
is elementary. If $H(c,x,y)=c^{N(x,y)}$ is its encoding from
Lemma~\ref{lem:uniform-power-encoding}, then
\[
  x\bmod y=
  \left\lfloor\frac{(x+2)H(x+1,x,y)}{H(x+2,x,y)}\right\rfloor.
\]
Indeed, the counted inequality is precisely
$B(B-1)^j\geq((x\bmod y)+1)B^j$, and
$B\geq(x\bmod y)+2$. This includes $y=0$. Here $H$ is a finite term
obtained by the encoding construction; the displayed bounded count
specifies that term's encoded value.
\end{remark}

\begin{corollary}\label{cor:binary-power-separation}
For every fixed integer $b\geq2$, binary exponentiation $x^y$ does not
belong to $S_b$.
\end{corollary}

\begin{proof}
Otherwise Theorem~\ref{thm:variable-basis} would imply
$\E=\mathcal T\subseteq S_b$, contradicting
Corollary~\ref{cor:fixed-base-incomplete}.
\end{proof}

\section{An encoded Hamming-weight function}\label{sec:hamming}

We now return to the fixed-base class $S$. Its predicate completeness
requires no use of the variable-base operation from the preceding section.
Recall that $\HW(x)$ is the number of ones in the binary expansion of $x$,
with $\HW(0)=0$. Although $\HW\notin S$ by
Corollary~\ref{cor:valuation-hamming}, we will construct $2^{2^{\HW(x)}}$ without constructing
$\HW(x)$ itself. The coefficient-extraction arguments are included to
check that no remainder or general subtraction is hidden in an $S$-term.

\begin{lemma}[A binomial coefficient as a residue]\label{lem:binomial}
For $0\leq b\leq a$, put $q=2^{2(a+2)}$ and
\begin{equation}\label{eq:binomial}
 U(a,b)=\left\lfloor
 \frac{q^{(a+1)^2+b+1}}{q^{a+2}-q-1}
 \right\rfloor.
\end{equation}
Then $U(a,b)\bmod q=\binom ab$, and both $U$ and $q$ are $S$-terms.
\end{lemma}

\begin{proof}
Every power $q^v$ displayed here is $2^{2(a+2)v}$ and belongs to $S$.
The denominator is $\pred(q\pred(q^{a+1}))$, so $U\in S$.
Let $c_j$ be the coefficient of $z^j$ in
\[
 \frac1{1-z^{a+1}-z^{a+2}}
 =\sum_{\ell\geq0}(z^{a+1}+z^{a+2})^\ell,
 \qquad m=a(a+1)+b.
\]
At degree $m$, only $\ell=a$ contributes: if $\ell<a$, the degree is at
most $(a-1)(a+2)<m$, and if $\ell>a$, it is at least $(a+1)^2>m$.
This also covers $a=0$, when the first alternative is empty. Thus
$c_m=\binom ab$.
For every $j\geq0$,
\[
 c_j\leq\sum_{\ell=0}^{\lfloor j/(a+1)\rfloor}2^\ell
 <2^{\lfloor j/(a+1)\rfloor+1}.
\]
For $j\leq m$, this gives $c_j<2^{a+1}<q$. For $s\geq1$,
$\lfloor(m+s)/(a+1)\rfloor\leq a+s$, so
\[
 0\leq\sum_{s\geq1}c_{m+s}q^{-s}
 <\frac{2^{a+2}}{q-2}<1.
\]
The real quotient in~\eqref{eq:binomial} is
$\sum_{j\geq0}c_jq^{m-j}$. Its floor is therefore
$\sum_{j=0}^m c_jq^{m-j}$, whose residue modulo $q$ is $c_m$.
\end{proof}

\begin{lemma}[An exponentiated greatest common divisor]\label{lem:gcd}
The function $\Gamma(a,b)=2^{\gcd(a,b)}$ belongs to $S$, with
$\gcd(0,0)=0$.
\end{lemma}

\begin{proof}
For positive $a,b$, put $h=2(ab+1)$, $q=2^h$, and
\[
 G(a,b)=\left\lfloor
 \frac{2^{h(ab+a+b)}}{\pred(2^{ha})\pred(2^{hb})}
 \right\rfloor\in S.
\]
The quotient before taking its floor equals
\[
 \frac{q^{ab}}{(1-q^{-a})(1-q^{-b})}
 =\sum_{i,j\geq0}q^{ab-ai-bj}.
\]
If $c_k$ counts the pairs $(i,j)\in\N^2$ satisfying $ai+bj=k$, then
$c_k\leq k+1$. Writing $m=ab$, the coefficients for $k\leq m$ are smaller
than $q=4^{m+1}$, and the fractional tail is bounded by
\[
 \sum_{s\geq1}(m+s+1)q^{-s}
 =\frac{m+1}{q-1}+\frac{q}{(q-1)^2}<1.
\]
For the last inequality, use $m\geq1$, $q\geq4(m+2)$, and $q\geq16$.
Consequently $G(a,b)\bmod q=c_{ab}$.
If $d=\gcd(a,b)$, the solutions at degree $ab$ are
\[
 i=\frac bd\,j',\qquad j=\frac ad(d-j'),\qquad 0\leq j'\leq d.
\]
Hence $c_{ab}=d+1$, and
\[
 \left\lfloor\frac{F(G(a,b),2^h)}2\right\rfloor=2^{\gcd(a,b)}.
\]
Call the term on the left $\Gamma_+(a,b)$, defined at all inputs by
our totalized operations. The single term
\[
 \delta_0(a)2^b+\sg(a)\delta_0(b)2^a
       +\sg(a)\sg(b)\Gamma_+(a,b)
\]
has the required value also on the axes and at $(0,0)$.
\end{proof}

\begin{lemma}[Lifting a visible residue]\label{lem:lift}
Let $U,Q,m\in S$, with $m>0$. There is a positive $S$-term $T$ such that
$T\equiv U\bmod Q\pmod m$.
\end{lemma}

\begin{proof}
Set $p=Q\lfloor U/Q\rfloor$ and $H=m(U+1)^2$. Then $0\leq p\leq U$
and $p^2<H$. Since
$H^2=(H-p)(H+p)+p^2$, the term
\[
 T=U+\left\lfloor\frac{H^2}{H+p}\right\rfloor
   =H+(U-p)=H+(U\bmod Q)
\]
is positive and has the claimed congruence. This includes $Q=0$.
\end{proof}

\begin{theorem}[Double Hamming exponential]\label{thm:hamming}
The function
\[
 D(x)=2^{2^{\HW(x)}}
\]
belongs to $S$.
\end{theorem}

\begin{proof}
Lemma~\ref{lem:binomial}, with $a=2x$ and $b=x$, gives
$\binom{2x}{x}=U(x)\bmod Q(x)$ for $S$-terms $U,Q$.
Lemma~\ref{lem:lift}, with $m=2^x$, gives a positive $S$-term $T$ with
$T\equiv\binom{2x}{x}\pmod{2^x}$. Thus
\[
 \gcd(T,2^x)=\gcd\!\left(\binom{2x}{x},2^x\right)=2^{\HW(x)}.
\]
For the second equality, the binary expansion gives
$\sum_{j\geq1}\lfloor x/2^j\rfloor=x-\HW(x)$. Applying the factorial
valuation formula to $(2x)!/(x!)^2$ yields
$\nu_2\binom{2x}{x}=2\HW(x)-\HW(2x)=\HW(x)$.
Also $\HW(x)\leq x$, including $x=0$.
Lemma~\ref{lem:gcd} now gives
$\Gamma(T,2^x)=2^{2^{\HW(x)}}\in S$.
\end{proof}

\section{Hypercube normal form and elementary predicates}\label{sec:predicates}

The hypercube method converts bounded zero existence into a Hamming-weight
comparison; see Mazzanti~\cite{Mazzanti2002}, Marchenkov~\cite{Marchenkov2007},
and Prunescu and Sauras-Altuzarra~\cite[Section~3]{PS2024}.
Its usual construction uses subtraction. We retain the difference of two
$S$-terms until the last step, then replace it by an $S$-term with a known
Hamming-weight correction. This is essential: representing the exponential
of a residue does not by itself justify arbitrary substitution of hidden
numerical values.

\subsection{Differences of terms and binary padding}

In this section, saying that an integer-valued function is a difference of
$S$-terms means only an identity $A-B$ with $A,B\in S$; it does not assert
that the difference belongs to $S$.

\begin{lemma}[Signed representations]\label{lem:signed}
Differences of $S$-terms are closed under finite sums, differences, and
products. For every fixed $j\in\N$, if $t,v\in S$ and $v\geq1$, then
\[
 \sum_{a=0}^{t-1}a^j2^{va}
\]
is a difference of $S$-terms. Here $0^0=1$ and a sum with $t=0$ is zero.
\end{lemma}

\begin{proof}
The first assertion follows by collecting signs, using in particular
\[
 (A-B)(C-D)=(AC+BD)-(AD+BC).
\]
For the second, write $q=2^v$ and formally differentiate the geometric sum:
\[
 \sum_{a=0}^{t-1}a^jq^a
 =\left(q\frac{d}{dq}\right)^j\frac{q^t-1}{q-1}
 =\frac{P_j(t,q,q^t)}{(q-1)^{j+1}}.
\]
There is a fixed polynomial $P_j\in\mathbb Z[t,q,Y]$. Explicitly,
$P_0=Y-1$, and the induction is
\[
 P_{j+1}=(q-1)\left(q\frac{\partial P_j}{\partial q}
                  +tY\frac{\partial P_j}{\partial Y}\right)
                  -(j+1)qP_j.
\]
This recurrence constructs a polynomial for each fixed $j$; no
differentiation or input-dependent iteration is an operation in $S$.
Split $P_j=P_j^+-P_j^-$ by its coefficient signs. The evaluations
$A=P_j^+(t,q,2^{vt})$, $B=P_j^-(t,q,2^{vt})$, and
$E=\pred(q)^{j+1}>0$ belong to $S$.
The quotient $(A-B)/E$ is an integer, so $A$ and $B$ have the same residue
modulo $E$. Therefore
\[
 \sum_{a=0}^{t-1}a^j2^{va}
 =\left\lfloor\frac AE\right\rfloor
  -\left\lfloor\frac BE\right\rfloor,
\]
which is the required representation.
\end{proof}

\begin{lemma}[Prescribed Hamming-weight correction]\label{lem:padding}
Suppose $A,B,u\in S$, $A\geq B$, and $u\geq1$. There is an $S$-term $V$
satisfying
\[
 \HW(V)=u+\HW(A-B).
\]
More explicitly, put
\begin{equation}\label{eq:padding}
 K=2(A+B+1),\qquad H=\pred(2^u)2^K,\qquad
 V=A+\left\lfloor\frac{H^2}{H+B}\right\rfloor.
\end{equation}
\end{lemma}

\begin{proof}
We have $H\geq2^K>B^2$ and $0\leq A-B\leq A<2^K$.
The identity $H^2=(H-B)(H+B)+B^2$ gives
$\lfloor H^2/(H+B)\rfloor=H-B$, hence
\[
 V=(2^u-1)2^K+(A-B).
\]
The summands occupy disjoint binary positions, and the first has exactly
$u$ ones. This proves the result.
\end{proof}

\subsection{Bounded equations for elementary predicates}

\begin{lemma}\label{lem:bounded-equations}
Given an elementary predicate $P(\vec n)$, one can effectively construct a
fixed $k\geq1$, terms $t,u\in S$ with $t\geq2$, $u\geq1$, and a
nonnegative expression $P_0(\vec n,\vec a)$ such that
\begin{align}
 P(\vec n)&\ \Longleftrightarrow\
 \exists\vec a\in\{0,\ldots,t-1\}^k\quad P_0(\vec n,\vec a)=0,
 \label{eq:bounded-existence}\\
 0&\leq P_0(\vec n,\vec a)<2^u
 \quad\text{throughout that cube}.\label{eq:bounded-size}
\end{align}
Moreover, $P_0$ is a finite signed sum of expressions
\begin{equation}\label{eq:simple-monomial}
 c(\vec n)\prod_{i=1}^k a_i^{\gamma_i}2^{e_i a_i},
\end{equation}
where $c\in S$ is nonnegative and $\gamma_i,e_i\in\N$ are fixed.
\end{lemma}

\begin{proof}
By~\eqref{eq:remainder-basis}, the characteristic function of $P$ has a
finite term using addition, remainder, and base-two exponentiation. Replace
its intermediate values by variables and impose equations for each gate.
For addition and exponentiation, use the nonnegative expressions
$(c-a-b)^2$ and $(c-2^a)^2$. For a remainder gate $c=a\bmod b$, introduce
nonnegative variables $q,s$ and use
\begin{equation}\label{eq:mod-gate}
 \bigl(b^2+(c-a)^2\bigr)
 \bigl((a-qb-c)^2+(b-c-s-1)^2\bigr).
\end{equation}
When $b=0$, the second factor cannot vanish, and the first vanishes exactly
when $c=a$. When $b>0$, the first factor cannot vanish, and the second
requires $a=qb+c$ and $b=c+s+1$, equivalently $c=a\bmod b$.

Let $P_0$ be the sum of the gate expressions and the square of the final
output minus one. If necessary, add a dummy variable with its square to
ensure $k\geq1$. Each summand is nonnegative, so $P_0=0$ is equivalent to
a correct evaluation with final value one. All coefficients obtained by
expanding are signed integer combinations of input monomials and powers
$2^{n_i}$, and every witness exponential has the form $2^{a_i}$.
Expansion therefore gives~\eqref{eq:simple-monomial}, with coefficient
signs placed outside the nonnegative functions $c\in S$.

The witnesses admit an $S$-term bound. Start with
$B_0=2+C+\sum_i n_i$, where $C$ is the sum of all constants appearing in
the fixed term, and repeat $B_{j+1}=2^{2B_j}$ once for each gate, with
one additional step. These finitely many steps give a term bounding all
intermediate values. For $b>0$, the required $q,s$ are at most $a,b$,
respectively; for $b=0$, choose $q=s=0$. Choose $t\geq2$ strictly larger
than this bound. Conversely, any zero in the cube represents a correct
evaluation, proving~\eqref{eq:bounded-existence}.

Finally, in the expansion~\eqref{eq:simple-monomial}, replace all signs by
plus signs and every $a_i$ by $t$. The resulting $S$-term $M$ bounds
$P_0$ on the cube. Taking $u=M+1$ proves~\eqref{eq:bounded-size}.
All constructions are effective from the initial term, and the basis
translation is effective from a standard elementary description.
\end{proof}

\subsection{The normal form and extraction of truth values}

\begin{theorem}[Hypercube normal form]\label{thm:hypercube}
For every elementary predicate $P(\vec n)$, one can effectively construct
$S$-terms $V,u,N$, with $u\geq1$, and a nonnegative integer-valued function
$d$ such that
\begin{equation}\label{eq:hypercube-normal}
 \HW(V)=u(N+d),\qquad P(\vec n)\ \Longleftrightarrow\ d>0.
\end{equation}
The function $d$ may be taken to count the zeros of the expression in
Lemma~\ref{lem:bounded-equations} in its bounded cube; no term for $d$ is
asserted.
\end{theorem}

\begin{proof}
Use $P_0,t,u,k$ from Lemma~\ref{lem:bounded-equations}, and let $d$ count
the zeros of $P_0$ in the cube. For $0\leq A<2^u$, set
\[
 \Delta(A,u)=(2^u-1)(2^u-A+1).
\]
If $A=0$, this is $2^{2u}-1$, with Hamming weight $2u$. If $A>0$, then
\[
 \Delta(A,u)=(2^u-A)2^u+(A-1),
\]
and the two $u$-bit blocks are complements, so the Hamming weight is $u$.
In all cases $0\leq\Delta(A,u)<2^{2u}$.
Consider the integer
\begin{equation}\label{eq:raw-hypercube}
 W=\sum_{0\leq a_1,\ldots,a_k<t}
 2^{2u(a_1+ta_2+\cdots+t^{k-1}a_k)}
 \Delta(P_0(\vec n,\vec a),u).
\end{equation}
The block indices enumerate $0,\ldots,t^k-1$. Thus the blocks do not
overlap, and
\begin{equation}\label{eq:raw-weight}
 \HW(W)=u(t^k+d).
\end{equation}

We first prove only that $W$ is a difference of two $S$-terms.
Expand $\Delta(P_0,u)=2^{2u}-2^uP_0+P_0-1$ and use the finite monomial
expansion in~\eqref{eq:simple-monomial}. Each sum in
\eqref{eq:raw-hypercube} then separates into a coefficient in $S$ times
a product of sums
\[
 \sum_{a_i=0}^{t-1}a_i^{\gamma_i}
      2^{(2ut^{i-1}+e_i)a_i}.
\]
The coefficient $2ut^{i-1}+e_i$ is a positive $S$-term. Lemma~\ref{lem:signed}
therefore gives $S$-terms $A,B$ with $W=A-B$. Since $W\geq0$, we have
$A\geq B$.
Apply Lemma~\ref{lem:padding} with this pair and the same $u$. Its term $V$
satisfies
\[
 \HW(V)=u+\HW(W)=u(t^k+1+d).
\]
Set $N=t^k+1\in S$. This is~\eqref{eq:hypercube-normal}, and
$P(\vec n)\Longleftrightarrow d>0$ follows from
\eqref{eq:bounded-existence}. All finite sums used to define $W$ have been
eliminated by Lemma~\ref{lem:signed}; the final construction is a fixed term.
\end{proof}

\begin{proof}[Proof of Theorem~\ref{thm:predicates}]
Take $V,u,N,d$ from Theorem~\ref{thm:hypercube}. By
Theorem~\ref{thm:hamming}, $D(V)=2^{2^{\HW(V)}}$ belongs to $S$.
Since $u\geq1$,
\[
 d>0\ \Longleftrightarrow\ \HW(V)>uN
       \ \Longleftrightarrow\ D(V)>2^{2^{uN}}.
\]
Consequently an $S$-term for the characteristic function is
\begin{equation}\label{eq:predicate-term}
 \chi_P(\vec n)=\sg\!\left(
 \left\lfloor\frac{D(V(\vec n))}{2^{2^{u(\vec n)N(\vec n)}}+1}
 \right\rfloor\right).
\end{equation}
If $d=0$, the numerator equals the denominator minus one. If $d>0$, it
is an integer strictly larger than that value, so the quotient is positive.
This proves $\Pred(\E)\subseteq\Pred(S)$. The reverse inclusion follows
from $S\subseteq\E$, established in Section~\ref{sec:separation}.
Every step was effective.

If an elementary $f$ has range contained in $\{0,\ldots,M\}$ for a fixed
constant $M$, each predicate $f=j$ has an $S$-characteristic term, and
\[
 f(\vec n)=\sum_{j=0}^M j\chi_{\{f=j\}}(\vec n)\in S.
\]
The sum has a fixed number of summands, so it is a term operation.
\end{proof}

\begin{remark}[Predicates versus numerical outputs]\label{rem:boundary}
Theorems~\ref{thm:main} and~\ref{thm:predicates} give
\[
 S\subsetneq\E,\qquad \Pred(S)=\Pred(\E).
\]
In particular, the predicates $z=x\bmod y$ and $x=2^z$ have characteristic
functions in $S$, although remainder and a decoder satisfying $T(2^n)=n$
do not belong to $S$. Predicate completeness permits arbitrary elementary
substitution at the level of truth values. It does not supply a term for
bounded search, a varying-length sum of characteristic values, or the
numerical witness to a true predicate. Nor does the proof establish
single-exponential closure for nested remainders; the direct construction
in Remark~\ref{rem:nesting} yields the double exponential
\eqref{eq:nested-encoding}. In contrast, the uniform variable-base
encoding of Lemma~\ref{lem:uniform-power-encoding} is closed under all
elementary substitutions, including nested remainders, and
Lemma~\ref{lem:geometric-recovery} converts those encodings to numerical
outputs. Thus predicate completeness of $S$ and function completeness of
$\mathcal T$ follow from different recovery mechanisms.
\end{remark}

\section*{AI Assistance Disclosure}
OpenAI's GPT-6.1 Sol assisted in solving the fixed-base open problem addressed in
this paper. The author presented the problem to the AI agent together
with private research findings and partial results. The agent initially
failed to resolve the problem. After the author suggested that the
principal obstruction might be the inability to lower exponents, the
agent developed the sparsity argument that led to the negative solution
proved here.

In subsequent work with GPT-6 Astra, the author proposed replacing
fixed-base exponentiation by the binary operation $x^y$. The suggestion
returned to an earlier attempt at a quotient basis, using hypercube
encoding inside of an exponent and a favorable rescaling before lowering the exponent via a geometric-series quotient.
 Following that direction, the agent developed
the uniform power-encoding argument, the geometric-series construction of
an encoded quotient, and the geometric recovery lemma. These three steps
prove the variable-base basis theorem. That proof proceeds without the
proposed hypercube and without binomial extraction. The agent also
extended the sparsity theorem uniformly to every fixed base, checked the
new constructions and their boundary cases, and incorporated the results
into the manuscript. The author supplied the original problem and its
research context, the earlier constructions, the exponent-lowering
diagnosis behind the negative result, and the variable-base and rescaling
direction behind the positive result.

\end{document}